\documentclass[12pt]{amsart}
\usepackage{amsmath, amscd}
\usepackage{amsfonts, graphicx}

\begin{document}
\numberwithin{equation}{section}

\def\1#1{\overline{#1}}
\def\2#1{\widetilde{#1}}
\def\3#1{\widehat{#1}}
\def\4#1{\mathbb{#1}}
\def\5#1{\frak{#1}}
\def\6#1{{\mathcal{#1}}}

\newcommand{\w}{\omega}
\newcommand{\Lie}[1]{\ensuremath{\mathfrak{#1}}}
\newcommand{\LieL}{\Lie{l}}
\newcommand{\LieH}{\Lie{h}}
\newcommand{\LieG}{\Lie{g}}
\newcommand{\de}{\partial}
\newcommand{\R}{\mathbb R}
\newcommand{\FH}{{\sf Fix}(H_p)}
\newcommand{\al}{\alpha}
\newcommand{\tr}{\widetilde{\rho}}
\newcommand{\tz}{\widetilde{\zeta}}
\newcommand{\tk}{\widetilde{C}}
\newcommand{\tv}{\widetilde{\varphi}}
\newcommand{\hv}{\hat{\varphi}}
\newcommand{\tu}{\tilde{u}}
\newcommand{\tF}{\tilde{F}}
\newcommand{\debar}{\overline{\de}}
\newcommand{\Z}{\mathbb Z}
\newcommand{\C}{\mathbb C}
\newcommand{\Po}{\mathbb P}
\newcommand{\zbar}{\overline{z}}
\newcommand{\G}{\mathcal{G}}
\newcommand{\So}{\mathcal{S}}
\newcommand{\Ko}{\mathcal{K}}
\newcommand{\U}{\mathcal{U}}
\newcommand{\B}{\mathbb B}
\newcommand{\NC}{\mathcal B_\beta}
\newcommand{\oB}{\overline{\mathbb B}}
\newcommand{\Cur}{\mathcal D}
\newcommand{\Dis}{\mathcal Dis}
\newcommand{\Levi}{\mathcal L}
\newcommand{\SP}{\mathcal SP}
\newcommand{\Sp}{\mathcal Q}
\newcommand{\A}{\mathcal O^{k+\alpha}(\overline{\mathbb D},\C^n)}
\newcommand{\CA}{\mathcal C^{k+\alpha}(\de{\mathbb D},\C^n)}
\newcommand{\Ma}{\mathcal M}
\newcommand{\Ac}{\mathcal O^{k+\alpha}(\overline{\mathbb D},\C^{n}\times\C^{n-1})}
\newcommand{\Acc}{\mathcal O^{k-1+\alpha}(\overline{\mathbb D},\C)}
\newcommand{\Acr}{\mathcal O^{k+\alpha}(\overline{\mathbb D},\R^{n})}
\newcommand{\Co}{\mathcal C}
\newcommand{\Hol}{{\sf Hol}(\mathbb H, \mathbb C)}
\newcommand{\Aut}{{\sf Aut}(\mathbb D)}
\newcommand{\D}{\mathbb D}
\newcommand{\oD}{\overline{\mathbb D}}
\newcommand{\oX}{\overline{X}}
\newcommand{\loc}{L^1_{\rm{loc}}}
\newcommand{\la}{\langle}
\newcommand{\ra}{\rangle}
\newcommand{\thh}{\tilde{h}}
\newcommand{\N}{\mathbb N}
\newcommand{\kd}{\kappa_D}
\newcommand{\Ha}{\mathbb H}
\newcommand{\ps}{{\sf Psh}}
\newcommand{\Hess}{{\sf Hess}}
\newcommand{\subh}{{\sf subh}}
\newcommand{\harm}{{\sf harm}}
\newcommand{\ph}{{\sf Ph}}
\newcommand{\tl}{\tilde{\lambda}}
\newcommand{\gdot}{\stackrel{\cdot}{g}}
\newcommand{\gddot}{\stackrel{\cdot\cdot}{g}}
\newcommand{\fdot}{\stackrel{\cdot}{f}}
\newcommand{\fddot}{\stackrel{\cdot\cdot}{f}}
\def\v{\varphi}
\def\Re{{\sf Re}\,}
\def\Im{{\sf Im}\,}
\def\rk{{\rm rank\,}}
\def\rg{{\sf rg}\,}
\def\Gen{{\sf Gen}(\D)}
\def\Pl{\mathcal P}

\newtheorem{theorem}{Theorem}[section]
\newtheorem{lemma}[theorem]{Lemma}
\newtheorem{proposition}[theorem]{Proposition}
\newtheorem{corollary}[theorem]{Corollary}

\theoremstyle{definition}
\newtheorem{definition}[theorem]{Definition}
\newtheorem{example}[theorem]{Example}

\theoremstyle{remark}
\newtheorem{remark}[theorem]{Remark}
\numberwithin{equation}{section}

\title[Regular poles and $\beta$-numbers]{Regular poles and $\beta$-numbers in the theory of holomorphic semigroups}
\author[F. Bracci]{Filippo Bracci}
\address{F. Bracci: Dipartimento Di Matematica\\
Universit\`{a} di Roma \textquotedblleft Tor Vergata\textquotedblright\ \\
Via Della Ricerca Scientifica 1, 00133 \\
Roma, Italy} \email{fbracci@mat.uniroma2.it}
\author[M.D. Contreras]{Manuel D. Contreras}
\address{M.D. Contreras \and  S. D\'{\i}az-Madrigal: Camino de los Descubrimientos, s/n\\
Departamento de Matem\'{a}tica Aplicada II \\
Escuela T\'ecnica Superior de Ingenier\'ia\\
Universidad de Sevilla\\
41092, Sevilla\\
Spain.} \email{contreras@us.es, madrigal@us.es}
\author[S. D\'{\i}az-Madrigal]{Santiago D\'{\i}az-Madrigal}
\thanks{Partially supported by the \textit{Ministerio
de Ciencia e Innovaci\'on} and the European Union (FEDER), project MTM2009-14694-C02-02, by the ESF Networking Programme ``Harmonic and Complex Analysis
and its Applications'' and by \textit{La Consejer\'{\i}a de Econom\'{\i}a, Innovaci\'{o}n y Ciencia
de la Junta de Andaluc\'{\i}a} (research group FQM-133)}

\date\today

\begin{abstract}
We introduce the notion of regular (boundary) poles for infinitesimal generators of semigroups of holomorphic self-maps of the unit disc. We characterize  such regular poles in terms of $\beta$-points ({\sl i.e.} pre-images of values with positive Carleson-Makarov $\beta$-numbers) of
the associated semigroup and of the associated K\"onigs intertwining function. We also define a natural duality operation in the cone of infinitesimal generators and show that the regular poles of an infinitesimal generator correspond to the regular null poles of the dual generator. Finally we apply such a construction to study radial multi-slits and give an example of a non-isolated radial slit whose tip has not a positive Carleson-Makarov $\beta$-number.
\end{abstract}



\maketitle

\section{Introduction}

The theory of semigroups of holomorphic self-maps of the unit disc has been developed  much in the past decades since the basic works of E. Berkson and H. Porta \cite{Berkson-Porta} and M. Heins \cite{H}  (see, {\sl e.g.}, \cite{Shb} and \cite{BCD2} for some recent accounts). Aside its own interest, the theory of semigroups  plays a fundamental r\^ole in  Loewner's theory (see, {\sl e.g.}, the recent paper \cite{BCD} where a general Loewner theory has been developed starting from semigroups theory). In the present paper we give a contribution to the general theory of semigroups by introducing and studying regular poles of infinitesimal generators.

Let $(\phi_t)$, $t\geq 0$, be a semigroup of holomorphic self-maps of the unit disc $\D$, generated by the holomorphic vector field $G$ (see Section \ref{semig} for definitions and properties). It is well known that, apart from the case $(\phi_t)$ is a group of elliptic automorphisms, there exists a unique point $\tau\in\de\D$, called the Denjoy-Wolff point of the semigroup, such that $(\phi_t)$ converges uniformly on compacta to the map $z\mapsto \tau$ as $t\to \infty$. Such a point is a ``regular zero'' of $G$, in the sense that $G(\tau)=0$ and $G'(\tau)\neq 0$ if $\tau \in \D$ and $\angle\lim_{z\to \tau} G(z)/(z-\tau)=L$ for some $L\leq 0$ if $\tau\in\de\D$ (here, as usual, $\angle\lim$ denotes non-tangential limit). Also, if $\tau\in \D$ then $\tau$ is the unique fixed point (in the interior of $\D$) of $\phi_t$ for all $t> 0$, while, if $\tau\in\de\D$ then $\tau$ is the unique boundary regular fixed point for $\phi_t$  with boundary dilatation coefficient less than or equal to $0$ for all $t>0$ (see Section \ref{boundaryreg}).

It is also known \cite[Theorem 1]{CDP2} (see also Section \ref{semig}) that the other boundary regular fixed points of $(\phi_t)$ correspond to  boundary regular null points of $G$.

Moreover, for a given semigroup $(\phi_t)$ there exists a (unique, once suitably normalized) univalent function $h:\D \to \C$ which simultaneously linearizes $\phi_t$ for all $t\geq 0$. Such a function is called the K\"onigs function of $(\phi_t)$ (see Section \ref{boundaryreg}). As the function $h$ intertwines $(\phi_t)$ with a linear semigroup of $h(\D)$, it is clear that the shape of $h(\D)$ reflects the dynamics of $(\phi_t)$. In particular, it is known (\cite{CD} and \cite{ESZ}) that a boundary regular fixed point  of $(\phi_t)$ corresponds to a direction going to infinity in $h(\D)$ contained in a fixed strip or in a fixed sector---depending on the displacement of the Denjoy-Wolff point.

There might be other types of singularities of $G$ on $\de \D$ and it is natural to expect them to have some dynamical meaning in terms of $(\phi_t)$. The aim of the present paper is to study ``regular poles'' of $G$, namely, points $x\in \de \D$ such that
\[
\angle\liminf_{z\to x}|G(z)(x-z)|=C,
\]
for some $C\in (0,+\infty)$ (see Section \ref{31}). In fact, as a consequence of the Berkson-Porta formula and a version of Julia's lemma due to C. Cowen and Ch. Pommerenke \cite{CP}, it turns out that, for all $y\in \de\D$ the limit $\angle\lim_{z\to y}G(z)(y-z)$ exists in $\C$, and we call its modulo the ``mass'' of the pole. Therefore, regular poles are in a sense the worst singularities an infinitesimal generator might have. Similarly to what happens for boundary regular null points of a given dilation (see \cite{CP}, \cite{EST}) we show that for any given $C>0$, there exist at most a finite number of regular poles of mass greater than or equal to $C$  (see Corollary~\ref{finito}).

Like boundary regular null points correspond to boundary regular fixed points,  also regular poles have  dynamical counterparts. Given a holomorphic map $f:\D\to \C$, we say that $x\in \de \D$ is a ``$\beta$-point'' (see Section \ref{32}) provided the limit
\begin{equation*}
\angle\limsup_{z\to x}|f'(z)|/|x-z|=L<+\infty.
\end{equation*}
We call $L$ the ``mass'' of the $\beta$-point $x$.

Our main result can then be stated as follows:

\begin{theorem}\label{main-intro}
Let $G$ be an infinitesimal generator and let $(\phi_t)$ be the associated semigroup of holomorphic self-maps of $\D$ and $h$ the associated K\"onigs function. Let $x\in \de \D$. Then the following are equivalent:
\begin{enumerate}
  \item $x$ is a regular pole of $G$,
  \item $x$ is a $\beta$-point for $\phi_t$ for some---hence for all---$t>0$,
  \item $x$ is a $\beta$-point for $h$.
\end{enumerate}
\end{theorem}

In fact, we have  much more quantitative and precise versions of the previous result (see Theorem \ref{main} and Theorem \ref{main2}).

As regular poles of $G$ correspond to poles of order $1$, and boundary regular null points of $G$ (different from the Denjoy-Wolff point) to zeros of order $1$, one can see them as a sort of ``dual concepts''. We make this idea precise in Section \ref{5}, where, using the Berkson-Porta representation of  infinitesimal generators, to each $\R$-semicomplete holomorphic vector field $G$ we naturally associate a unique infinitesimal generator $\hat{G}$ in such a way that boundary regular null points of $G$ correspond to regular poles of $\hat{G}$. Using such duality, in Proposition \ref{GGhat} we give a precise and finer characterization of boundary regular fixed points of semigroups.

$\beta$-Points of univalent functions are connected---and that's the reason we chose such a name---to the ``$\beta$-numbers'' introduced using extremal length by L. Carleson and N. Makarov \cite{CM}.  In fact, D. Bertilsson \cite[Theorem 2.1]{Be},  showed that given a univalent function $f:\D \to \C$, the points of $\de(f(\D))$ with positive Carleson-Makarov $\beta$-numbers correspond one-to-one to the $\beta$-points of $f$ and the $\beta$-number of a point in $\de(f(\D))$, if positive, is equal to the reciprocal of the mass of the corresponding $\beta$-point  up to a non-vanishing factor.

The  $\beta$-numbers are related to  the Brennan conjecture. Referring the reader to \cite{Be} for a detailed account on the topic, here we content ourselves to say that the Brennan conjecture can be reformulated in terms of a (universal) bound on the sum of  $\beta$-numbers of any given univalent function. Although the Brennan conjecture has been solved for star-like and close-to-convex functions (the classes  which K\"onigs functions belong to) by B. Dahlberg and J. Lewis (see \cite{B}), the question of what the shape of a simply connected domain looks like near a point with positive $\beta$-number  is still open. In his PhD thesis, Bertilsson gives some necessary and some sufficient conditions of geometric flavor for a point to have positive $\beta$-number. In Section \ref{beta} we examine the shape of the image of K\"onigs functions near  the image of a $\beta$-point from a measure theoretic point of view, relating  the positive Borel measure  in the Herglotz representation formula of the infinitesimal generator to $\beta$-points (see Proposition \ref{mu-nc}).  Out of this, we give some sufficient conditions of geometric character for a point to be a $\beta$-point. In particular, as it is well known, the tip of an ``isolated radial slit'' always corresponds to a $\beta$-point, but, in Example \ref{no-tip} we show that the tip of a non-isolated radial slit might not have positive Carleson-Makarov $\beta$-number. Such a construction is based on a representation formula for holomorphic vector fields which generate a radial $m$-slits evolution in the complex plane  and on a geometrical interpretation of the terms arising, made by using our duality (see Proposition \ref{forma-pradial}).

\medbreak

This work started while the first and third named authors were visiting the Mittag-Leffler Institute, during the program ``Complex Analysis and Integrable Systems'' in Fall 2011 and it was completed while the first named author was visiting the Departamento de Matem\'{a}tica Aplicada II in Seville. The authors thank both the
 Mittag-Leffler  Institute and the University of Seville for the kind hospitality and the atmosphere experienced there.

\section{Preliminaries}
\subsection{Boundary regular contact and fixed points}\label{boundaryreg}

For the unproven statements, we refer the reader to, {\sl e.g.}, \cite{Abate}, \cite{CMbook} or \cite{Shb}.

Let $f:\D \to \D$ be holomorphic, $x\in \de \D$, and let
\[
\al_x(f):=\liminf_{z\to x} \frac{1-|f(z)|}{1-|z|}.
\]
By Julia's lemma, it follows that $\al_x(f)>0$. The number $\al_x(f)$ is called the {\sl boundary dilatation coefficient} of $f$.

If $f:\D \to \C$ is a map and $x\in \de\D$, we write $\angle\lim_{z\to x}f(z)$ for the non-tangential (or angular) limit of $f$ at $x$.

In the following we will make use of this version of Julia-Wolff-Carath\'eodory's theorem:
\begin{theorem}[Julia-Wolff-Carath\'eodory]\label{JWC}
Let $f: \D \to \D$ be holomorphic. Let $x\in \de \D$ and assume that
\[
\limsup_{(0,1)\ni r\to 1}|f(rx)|=1.
\]
Then
\[
\al_f(x)=\limsup_{(0,1)\ni r\to 1}|f'(rx)|.
\]
 Moreover, if $\al_f(x)<+\infty$ then there exists $y\in \de \D$ such that
\[
\angle \lim_{z\to x}f(z)=y, \quad \angle \lim_{z\to x} |f'(z)|=\angle\lim_{z\to x} \left|\frac{y-f(z)}{x-z}\right|=\al_f(x).
\]
\end{theorem}

\begin{definition}
Let $f:\D \to \D$ be holomorphic. A point $x\in \de \D$ is a {\sl contact point} if $\angle \lim_{z\to x}f(z)=y\in \de\D$. The point $x\in \de \D$ is a {\sl regular contact point} for $f$ if $\al_f(x)<+\infty$. A point $x\in\de \D$ is {\sl boundary regular fixed point} for $f$ if $x$ is a regular contact point and $\angle \lim_{z\to x}f(z)=x$.
\end{definition}

\begin{remark}
By the Julia-Wolff-Carath\'eodory Theorem \ref{JWC}  a regular contact point is in fact a contact point.
\end{remark}

If $x\in \de \D$ is a contact point for $f$, as customary, we let
\[
f(x):=\angle\lim_{z\to x}f(x).
\]

If $f:\D\to\D$ is holomorphic, not the identity nor an elliptic automorphism, by the Denjoy-Wolff theorem, there exists a unique point $\tau\in\oD$, called the {\sl Denjoy-Wolff point} of $f$, such that $f(\tau)=\tau$ and the sequence of iterates $\{f^{\circ k}\}$ converges uniformly on compacta of $\D$ to the constant map $z\mapsto \tau$. If the Denjoy-Wolff point $\tau$ of $f$ belongs to $\D$ (and $f$ is not the identity nor an elliptic automorphism) then by Schwarz lemma $|f'(\tau)|<1$, while, if $\tau\in\de\D$, then $\al_f(\tau)\leq 1$.

\begin{remark}\label{magg}
By the Wolff lemma, if $x$ is a regular contact point for $f$ which is not the Denjoy-Wolff point of $f$, then $\al_f(x)>1$.
\end{remark}

\begin{lemma}\label{L1}
Let $f,g: \D \to \D$ be holomorphic. Let $h:=f\circ g$. Suppose that $x\in \de \D$ is a (regular) contact point for $h$. Then $x$ is a (regular) contact point for $g$ and the point $g(x)\in \de \D$ is a (regular) contact point for $f$. Moreover, $h(x)=\angle \lim_{z\to g(x)}f(z)$ and $\al_h(x)=\al_f(g(x))\cdot \al_g(x)$.
\end{lemma}

\begin{proof}
Let $\gamma(r):=g(rx)$, for $r\in (0,1)$. Let $\Omega_\gamma$ be the  $\omega$-limit of $\gamma$, namely, $q\in \Omega_\gamma$ if there exists $\{r_k\}\subset (0,1)$ converging to $1$ such that $\gamma(r_k)\to q$. We claim that there exists $\sigma\in \de \D$ such that  $\Omega_\gamma=\{\sigma\}$. Indeed, if $q\in \Omega_\gamma\cap \D$, and $\{r_k\}\subset (0,1)$ is a sequence converging to $1$ such that $\gamma(r_k)\to q$, then $h(r_k)=f(\gamma(r_k))\to f(q)\in \D$, contradicting the hypothesis that $h$ has a contact point at $x$. Therefore $\Omega_\gamma$ is a connected and compact subset of $\subset \de \D$, namely, it is a closed arc $A$, possibly reducing to a point. Suppose that $A$ is not a point. Thus $\lambda(A)>0$, where $\lambda$ is the Lebesgue measure on $\de \D$. By Fatou's theorem, there exists a subset $A'\subseteq A$ such that  $\lambda(A)=\lambda(A')$ and $f$ has radial limit at all points $q\in A'$. Let $q\in A'$, $q\not\in\de A$ (where the boundary is taken in $\de \D$). Since $q\in A=\Omega_\gamma$ but it is not an extreme of the arc $A$, then the radial segment $\Gamma_q:=\{sq: s\in (0,1)\}$ intersects the curve $\gamma$ infinitely many times, that is, there exists a sequence $\{r_{k}\}\subset (0,1)$ converging to $1$ and such that $\gamma(r_k)\in \Gamma_q$. Therefore
\[
\lim_{(0,1)\ni r\to 1}f(r)=\lim_{k\to \infty} f(\gamma(r_k))=\lim_{k\to \infty} h(r_k)=h(x).
\]
Hence, again by Fatou's theorem, $f\equiv h(x)$, a contradiction. Thus $A$ reduces to a point $\sigma\in \de \D$. This shows that $g$ has radial limit $\sigma\in \de \D$ and Lindel\"of's theorem implies that
\[
\angle \lim_{z\to x}g(z)=\sigma=:g(x).
\]
By the same token, $f$ has limit $h(x)$ along the curve $(0,1)\ni r\mapsto g(r)$ which converges to $g(x)$, hence it has non-tangential limit $h(x)$ at $g(x)$.

Now,
\begin{equation*}
\left|\frac{1-h(r)}{1-r} \right|=\left|\frac{1-f(g(r))}{1-g(r)} \right|\left|\frac{1-g(r)}{1-r} \right|,
\end{equation*}
and the rest of the statement follows from the Julia-Wolff-Carath\'eodory Theorem \ref{JWC}.
\end{proof}

\subsection{Semigroups and infinitesimal generators}\label{semig} A semigroup $(\phi_t)$ of holomorphic
self-maps of $\D$ is a continuous homomorphism between the
additive semigroup $(\R^+, +)$ of positive real numbers and the
semigroup $({\sf Hol}(\D,\D),\circ)$ of holomorphic self-maps
of $\D$ with respect to the composition, endowed with the
topology of uniform convergence on compacta.

By Berkson-Porta's theorem \cite{Berkson-Porta}, if $(\phi_t)$
is a semigroup in ${\sf Hol}(\D,\D)$ then $t\mapsto \phi_t(z)$
is analytic and there exists a unique holomorphic vector field
$G:\D\to \C$ such that
\[
\frac{\de \phi_t(z)}{\de
t}=G(\phi_t(z)).
\]
Such a vector field $G$, called the  {\sl infinitesimal generator} of $(\phi_t)$,  is {\sl semicomplete} in the sense that the Cauchy problem
\[
\begin{cases}
\stackrel{\bullet}{x}=G(x(t))\\
x(0)=z
\end{cases}
\]
has a  solution $x^z:[0,+\infty)\to \D$ for all $z\in \D$.
Conversely, any semicomplete holomorphic vector field in $\D$
generates a semigroup in ${\sf Hol}(\D,\D)$.

We denote by $\Gen$ the set of infinitesimal generators in $\D$. Recall that $\Gen$ is a closed (in ${\sf Hol}(\D,\C)$) convex cone with vertex in $0$.

Let $G\not\equiv 0$ be an infinitesimal generator with
associated semigroup $(\phi_t)$. Then there exists a unique
$\tau\in\oD$ and a unique $p:\D\to \C$ holomorphic with $\Re
p(z)\geq 0$ such that the following formula, known as the {\sl Berkson-Porta
formula}, holds
\[
G(z)=(z-\tau)(\overline{\tau}z-1)p(z).
\]
The point $\tau$ in the Berkson-Porta formula turns out to be
the  Denjoy-Wolff point of $\phi_t$ for all $t> 0$.
Moreover, if $\tau\in \de\D$ it follows $\angle\lim_{z\to
\tau}\phi_t'(z)=e^{\beta t}$ for some $\beta\leq 0$.

A {\sl boundary regular fixed point}  for a semigroup
$(\phi_t)$ is a point $p\in \de \D$ which is a boundary regular fixed point  all $\phi_t$, $t>0$.

A {\sl boundary regular null point} for an infinitesimal generator $G$, is a point $x\in \de \D$ such that
\[
\angle\lim_{z\to x}\frac{G(z)}{z-x}=\ell\in \R,
\]
exists finite. The number $\ell$ is called the {\sl dilation} of $G$ at $x$.

The following result shows the relations among the various objects introduced so far. The statements are taken from  \cite[Theorem 1]{CDP}, \cite[Theorem 2]{CDP2} (except the first claim about boundary regular points  which is essentially in \cite[pag. 255]{Siskakis-tesis}), see also \cite{ES}.

\begin{proposition}\label{bsem}
Let $(\phi_t)$ be a semigroup of holomorphic self-maps of $\D$ with infinitesimal generator $G$. Let $x\in \de \D$. If $x$ is a boundary (regular) fixed point for $\phi_{t_0}$ for some $t_0>0$ then it is a boundary (regular) fixed point for $\phi_t$ for all $t\geq 0$. Moreover, the following are equivalent:
\begin{enumerate}
  \item $x$ is a boundary regular fixed point for $(\phi_{t})$  and the   boundary dilation coefficient of $\phi_t$ at $x$ is $e^{\beta t}$ for some $\beta>0$,
  \item $x$ is a boundary regular null point for $G$ with dilation $\beta>0$.
\end{enumerate}
\end{proposition}

To any semigroup of holomorphic self-maps of the unit disc is associated a (unique) intertwining map which simultaneously linearizes the semigroup. The ideas for the proof of the following result are in \cite{H}, and, with different methods in \cite{Berkson-Porta} and \cite{Siskakis-tesis} (see also \cite[Chapter 1.4]{Abate}).

\begin{proposition}
\label{Unival-VectorField} Let $(\phi_{t})$ be a non-trivial
semigroup in $\mathbb{D}$ with infinitesimal generator $G$.
Then there exists a unique univalent function $h:\D\to\C$,
called the {\sl K\"onigs function} of $(\phi_t)$, such that
\begin{enumerate}
\item If $(\phi_{t})$ has Denjoy-Wolff point $\tau \in \mathbb{D}$ then
$h(\tau)=0$, $h'(\tau)=1$ and
$h(\phi_t(z))=e^{G'(\tau)t}h(z)$ for all $t\geq 0$.
Moreover, $h$ is the unique holomorphic function from
$\mathbb{D}$ into $\mathbb{C}$ such that
\begin{enumerate}
\item[(i)] $h^{\prime }(z)\neq 0,$ for every $z\in \mathbb{D},$
\item[(ii)] $h(\tau )=0$ and $h^{\prime }(\tau )=1,$
\item[(iii)] $h^{\prime }(z)G(z)=G^{\prime }(\tau )h(z),$ for
every $z\in \mathbb{D}.$
\end{enumerate}
\item If $(\phi_{t})$ has Denjoy-Wolff point $\tau
\in \partial \mathbb{D}$ then $h(0)=0$ and
$h(\phi_t(z))=h(z)+t$ for all $t\geq 0$. Moreover, $h$ is
the unique holomorphic function from $\mathbb{D}$ into
$\mathbb{C}$ such  that:
\begin{enumerate}
\item[(i)] $h(0)=0,$
\item[(ii)] $h^{\prime }(z)G(z)=1,$ for every $z\in \mathbb{D}.$
\end{enumerate}
\end{enumerate}
\end{proposition}

Boundary regular fixed points of semigroups can be detected by looking at the geometry of the image of the associated K\"onigs function. The proof of the following lemma is essentially contained in \cite[Theorem 2.6]{CD} and \cite[Lemma 5]{ESZ}.

\begin{lemma}\label{angulo}
Let $\Omega\subset \C$ be a domain star-like with respect to $0$. Let $h:\D\to \Omega$ be the Riemann map such that $h(0)=0, h'(0)>0$. For $\eta\in  [0,1)$ and $\al\in (0,1/2)$ let
$V_\eta(\al):=\{z\in \C : z=s e^{2\pi i\theta}: s\in [0,+\infty), \theta\in (\eta-\al,\eta+\al)\}$.

Assume there exists $\eta_0\in  [0,1)$ and $\al_0\in (0,1/2)$ such that $V_{\eta_0}(\al_0)\subset \Omega$. Then
$x:=\lim_{\R \ni r\to +\infty} h^{-1}(re^{2\pi i \eta})\in \de \D$ is a boundary regular fixed point of
the semigroup $\phi_t(z):=h^{-1}(e^{-t}h(z))$ with dilation  $\beta$ given by
\[
  \beta=\sup \{\al : V_\eta(\al)\subset \Omega \hbox{ and } V_{\eta_0}(\al_0)\subset V_\eta(\al)\}.
\]
\end{lemma}

\section{Regular poles and $\beta$-points}\label{3}

\subsection{Regular poles}\label{31}
\begin{definition}
Let $G\in \Gen$.  A point $x\in \de \D$ is a {\sl  regular pole of $G$ of  mass $C>0$} if
\[
\angle\liminf_{z\to x} |G(z)(x-z)|= C.
\]
We denote by $\Pl_C(G)$ the set of  regular poles of $G$ of  mass $C$. Moreover, we let
\[
\Pl(G):=\cup_{C>0} \Pl_C(G)
\]
be the set of {\sl regular poles} of $G$.
\end{definition}

As we see, infinitesimal generators behave well at regular poles, essentially as a consequence of  Julia's lemma, whose following version  was first proved in \cite[Lemma 4.0]{CP} (see also \cite[Lemma 4.2]{EST}):

\begin{lemma}[Cowen-Pommerenke]\label{Co-Po}
Let $p:\D \to \C$ be holomorphic and $\Re p(z)\geq 0$ for all $z\in \D$. Then for all $x\in\de \D$ the following limit exists
\[
\angle\lim_{z\to x} \frac{1}{2}p(z)(1-\overline{x}z)=L \in [0,+\infty).
\]
Moreover, the function $h(z):=p(z)-L (x+z)/(x-z)$ is such that $\Re h(z)\geq 0$.
\end{lemma}

Berkson-Porta's formula and  Lemma \ref{Co-Po} imply:

\begin{lemma}\label{exists}
If $G\in \Gen$, then for all $x\in \de \D$  the non-tangential limit
\begin{equation*}
\angle\lim_{z\to x} G(z)(x-z)=a
\end{equation*}
exists, with $|a|\in [0,+\infty)$.
\end{lemma}

\begin{remark}
By Lemma \ref{exists}, if $G\in \Gen$ then
\[
\Pl_C(G)=\{x\in \de \D : \angle\lim_{z\to x} |G(z)(x-z)|=C\}.
\]
\end{remark}

Note that if $x\in \de \D$ is a regular pole then
\[
\angle \lim_{z\to x} |G(z)|=\infty.
\]

Along the lines of the Cowen-Pommerenke inequalities \cite{CP} (see also \cite{EST}) we prove the following result:

\begin{proposition}
Let $G\in \Gen$ be given by the Berkson-Porta formula $G(z)=(\tau-z)(1-\overline{\tau}z)p(z)$ for some $\tau\in \oD$ and $p:\D \to \C$ holomorphic with $\Re p(z)\geq 0$ for all $z\in \D$. Let $A_j>0$, $j=1,\ldots, m$. Let $x_j\in \Pl_{A_j}(G)$ for $j=1,\ldots, m$. Then
\[
\sum_{j=1}^m \frac{A_j}{2|x_j-\tau|^2}\leq \Re p(0).
\]
\end{proposition}

\begin{proof} We can suppose that $p$ is not constant.
For $j=1,\ldots, m$, let
\[
L_j:=\lim_{(0,1)\ni r\to 1} \frac{1-r}{2}p(rx_j).
\]
By Lemma \ref{Co-Po} such a limit exists finite and non-negative and moreover the function
\[
h(z):=p(z)-\sum_{j=1}^m L_j \frac{x_j+z}{x_j-z}
\]
is holomorphic in $\D$ and $\Re h(z)\geq 0$ for all $z\in \D$. From $\Re h(0)\geq 0$, taking into account that
\[
A_j=\lim_{(0,1)\ni r\to 1}|G(rx_j)|(1-r)=2|x_j-\tau|^2L_j,
\]
we have the inequality.
\end{proof}

A direct consequence of the previous proposition is the following:

\begin{corollary}\label{finito}
Let $G\in \Gen$. Then for all $C>0$
\[
\sharp \left(\cup_{D\geq C}\Pl_D(G)\right)<+\infty.
\]
\end{corollary}

\subsection{$\beta$-Points}\label{32} We start with some general facts. First, we recall the following (for a proof see, {\sl e.g.}, \cite[Theorem 10.5]{P}):

\begin{lemma}\label{second-der}
Let $f:\D \to \C$ be holomorphic and let $x\in \de\D$. The following are equivalent:
\begin{enumerate}
  \item $\angle\lim_{z\to x}f(z)=A\in \C$ and
  \[
  \angle\lim_{z\to x} \frac{f(z)-A}{z-x}=L\in \C.
  \]
  \item $\angle\lim_{z\to x} f'(z)=L\in \C$.
\end{enumerate}
\end{lemma}

In the sequel we will strongly  use the following fact:

\begin{lemma}\label{noboundary}
Let $f: \D \to \D$ be univalent. Let $x\in \de \D$. Suppose that
\[
\lim_{(0,1)\ni r\to 1}|f'(rx)|=0.
\]
 Then
there exists $\sigma\in \D$ such that $\angle\lim_{z\to x}f(z)=\sigma$.
\end{lemma}

\begin{proof}
Since $f$ is univalent, by the Lehto-Virtanen theorem (see, {\sl e.g.} \cite[Lemma 9.3, Theorem 9.3]{P}), our hypothesis implies that
 \begin{equation}\label{ipogood}
\angle\lim_{z\to x}|f'(z)|=0.
\end{equation}
Now, let $K(x,R):=\{z\in \D: |x-z|<R(1-|z|)\}$, for some $R>1$, be a Stolz angle. By Lemma \ref{second-der} there exists $\sigma\in \oD$ such that $\angle \lim_{z\to x} f(z)=\sigma$.

Now, if $\sigma\in \de\D$,  the Julia-Wolff-Carath\'eodory Theorem \ref{JWC} implies that
\[
0= \lim_{(0,1)\ni r\to 1}|f'(rx)|=\al_f(x)>0,
\]
a contradiction.
\end{proof}

\begin{remark}
The previous lemma holds for holomorphic self-maps of the unit disc which are not necessarily univalent replacing the hypothesis  with \eqref{ipogood}.
\end{remark}

\begin{definition}
Let $f:\D \to \C$ be holomorphic. A point $x\in \de \D$ is a {\sl $\beta$-point}  if
\[
\angle \limsup_{z\to x} \frac{|f'(z)|}{|x-z|}=L<+\infty.
\]
We call $L$ the {\sl mass} of $f$ at $x$ and we denote by $\NC(f)$ the set of all $\beta$-points  of $f$.
\end{definition}

Note that if $x\in \NC(f)$ then $\angle \lim_{z\to x} |f'(z)|=0$.

\begin{remark}
By \cite[Theorem 2.1]{Be}, if $f:\D \to \C$ is univalent, for every $x\in \de \D$ the angular limit
\[
\angle\lim_{z\to x} \frac{|f'(z)|}{|x-z|}=L
\]
exists, with $L\in (0, +\infty]$. However, we will not use this result in here.
\end{remark}

\section{The main results}\label{4}  We are going to prove that regular poles of infinitesimal generators correspond to $\beta$-points of associated semigroups, plus some other simpler characterizations:

\begin{theorem}\label{main}
Let $G\in \Gen$ and let $(\phi_t)$ be the associated semigroup of holomorphic self-maps of $\D$. The following are equivalent:
\begin{enumerate}
\item $\displaystyle{\angle\limsup_{(0,1)\ni r\mapsto 1} |G(rx)|(1-r)>0},$
  \item $x\in \Pl(G)$,
  \item $x\in \NC(\phi_t)$ for some---hence for all---$t>0$,
  \item $\angle\lim_{z\to x}\phi_t'(z)=0$,  $\angle\lim_{z\to x}\phi_t''(z)=L\in \C$ for some---and hence for all---$t>0$,
  \item $\lim_{(0,1)\ni r\mapsto 1} |\phi_t'(rx)|=0$, $\limsup_{(0,1)\ni r\mapsto 1} |\phi_t''(rx)|<+\infty$ for some---and hence for all---$t>0$,
  \item there exists $t_0>0$ such that $\displaystyle{\limsup_{(0,1)\ni r\mapsto 1}\frac{|\phi_{t_0}'(rx)|}{1-r}<+\infty}$
\end{enumerate}
Moreover, if the previous conditions are satisfied then for all $t>0$
\begin{itemize}
  \item[a)] $\angle \lim_{z\to x}\phi_t(z)=\sigma_t\in \D$, $\angle \lim_{z\to x}\phi'_t(z)=0$,
  \item[b)] if $x\in \Pl_{|A|}(G)$ with $A=\angle\lim_{z\to x} G(z)(z-x)$ then
  \[
  \angle\lim_{z\to x}\phi_t''(z)=\angle \lim_{z\to x}\frac{\phi_t'(z)}{z-x}=\frac{G(\sigma_t)}{A}.
  \]
\end{itemize}
\end{theorem}

\begin{proof}
By Lemma \ref{exists}, (1) is equivalent to (2). Moreover, clearly, (3) implies (6). While (4) implies (5).

Next, assume (6) holds. By Lemma \ref{noboundary} there exists $\sigma\in \D$ such that $\angle\lim_{z\to x}\phi_{t_0}(z)=\sigma$.
Recall that, differentiating in $s$ the equality $\phi_{t}(\phi_s(z))=\phi_{t+s}(z)$ at $s=0$, for all $t\geq 0$ and $z\in \D$ we have
\begin{equation}\label{diffeq}
\phi_t'(z)G(z)=G(\phi_t(z)).
\end{equation}
Thus, for $r\in (0,1)$, we obtain
\begin{equation}\label{equo1}
\frac{|\phi'_{t_0}(rx)|}{1-r}|G(rx)|(1-r)=|G(\phi_{t_0}(rx))|.
\end{equation}
Now, $\lim_{r\to 1}G(\phi_{t_0}(rx))=G(\sigma)\in \C$.
If  $G(\sigma)\neq 0$ then (1) follows immediately from \eqref{equo1}. But, if  $G(\sigma)=0$ then $\phi_{t_0}(\sigma)=\sigma$. The function $\phi_{t_0}$ is univalent, thus for any open disc $U$ relatively compact in $\D$ and containing $\sigma$ there exists an open set  $V$ relatively compact in $\D$ and containing $\sigma$  such that $\phi_{t_0}(U)=V$ and $\phi_{t_0}(\D\setminus U)\subset \D\setminus V$. Since $\{rx\}$ is eventually outside  any relatively compact disc containing $\sigma$, it follows $\phi_{t_0}(rx)\not\to \sigma$, a contradiction. Thus $G(\sigma)\neq 0$.

Now, assume that (2) holds. In this case we cannot argue directly using \eqref{equo1} because, {\sl a priori}, $r\mapsto \phi_{t_0}(rx)$ might have (even tangential) limit at a boundary point where $G$ explodes. Let
\[
C:=\angle\lim_{z\to x} |G(z)(x-z)|>0.
\]
For $t\geq 0$, let
\[
A(t):=\limsup_{(0,1)\ni r\mapsto 1}|\phi_t'(rx)|.
\]
We claim that
\begin{equation}\label{claim1}
A(t)=0, \quad \forall t> 0.
\end{equation}
Assume for the moment that \eqref{claim1} is true. Fix $t>0$. Then Lemma \ref{noboundary} implies that there exists $\sigma_t\in \D$ such that $\angle\lim_{z\to x}\phi_t(z)=\sigma_t$. By \eqref{diffeq}
\[
\angle \lim_{z\to x}\frac{|\phi_t'(z)|}{|x-z|}=\angle\lim_{z\to x}\frac{|G(\phi_t(z))|}{|G(z)||x-z|}=\frac{|G(\sigma_t)|}{C}<+\infty,
\]
proving that $x\in \NC(\phi_t)$ for all $t>0$, that is (3).

We are left to show that \eqref{claim1} holds. Arguing by contradiction, let $t_0>0$ be such that $A(t_0)>0$. Let $\{r_n\}\subset (0,1)$ be a sequence converging to $1$ such that
\[
\lim_{n\to \infty} |\phi_{t_0}'(r_nx)|=A(t_0).
\]
Up to subsequences, we can also assume that $\phi_{t_0}(r_nx)\to \sigma\in \oD$. By \eqref{diffeq} we have for $n\in \N$
\begin{equation}\label{limited}
|\phi_{t_0}'(r_nx)|=\frac{|G(\phi_{t_0}(r_nx))||1-r_n|}{|G(r_nx)||1-r_n|}.
\end{equation}
Now, $|G(r_nx)||1-r_n|\to C>0$ by hypothesis. Thus, if $\sigma\in \D$ then $|G(\phi_{t_0}(r_nx))||1-r_n|\to 0$, which implies $A(t_0)=0$, against our hypothesis on $A(t_0)$. Hence $\sigma\in \de \D$.

Let $G(z)=(\tau-z)(1-\overline{\tau}z)p(z)$ be the Berkson-Porta decomposition of $G$. Let $g(z):=-zp(\phi_{t_0}(z))$.  Again by the Berkson-Porta formula,  $g\in \Gen$. By Lemma \ref{exists} it follows that
\[
\lim_{n\to \infty}|g(r_n x)||1-r_n|=b<+\infty.
\]
Hence there exists $M>0$ such that
\[
|G(\phi_{t_0}(r_nx))||1-r_n|=|g(r_n x)||1-r_n|\frac{|\tau-\phi_{t_0}(r_nx)||1-\overline{\tau}\phi_{t_0}(r_nx)|}{r_n}\leq M
\]
for all $n\in \N$. Thus \eqref{limited} shows that
\[
0<A(t_0)<+\infty.
\]
By the Julia-Wolff-Carath\'eodory Theorem \ref{JWC}, it follows that $\angle \lim_{z\to x} \phi_{t_0}(z)=\sigma$ and $\angle\lim_{z\to x}|\phi_{t_0}'(z)|=A(t_0)$. Hence, $x$ is a regular contact point for $\phi_{t_0}$. If $x$ were a boundary regular fixed point for $\phi_{t_0}$ then by Proposition \ref{bsem} it would hold $\angle \lim_{z\to x}|G(z)|=0$ against $x\in \Pl(G)$.

Therefore $x$ is a regular contact point for $\phi_{t_0}$ which is not  fixed. We claim that $x$ is a regular contact point for $\phi_{t_0/2^n}$ for all $n\in \N$ and that the points $\sigma_n:=\phi_{t_0/2^n}(x)$ are such that $\sigma_m\neq \sigma_n$ for $m\neq n$.

Since $\phi_{t_0}=\phi_{t_0/2}\circ \phi_{t_0/2}$, we can apply Lemma \ref{L1}, which implies that $x$ is a regular contact point for $\phi_{t_0/2}$. Moreover, if $\sigma_2:=\angle \lim_{z\to x}\phi_{t_0/2}(z)$ then $\sigma_2$ is a regular contact point for $\phi_{t_0/2}$ and $\phi_{t_0/2}(\sigma_2)=\phi_{t_0}(x)$. By induction we prove that $x$ is a regular contact point for $\phi_{t_0/2^n}$ for every $n\in \N$.

Now we have to show that $\sigma_n\neq \sigma_m$ for $n\neq m$. Assume this is not the case and let $\sigma_n=\sigma_m$ for some $n>m$. Then
\[
\phi_{\frac{t_0}{2^n}(2^{n-m}-1)}(\sigma_n)=\phi_{\frac{t_0}{2^m}-\frac{t_0}{2^n}}(\phi_{\frac{t_0}{2^n}}(x))=\phi_{\frac{t_0}{2^m}}(x)=\sigma_{m}=\sigma_n.
\]
Hence, $\phi_{\frac{t_0}{2^n}(2^{n-m}-1)}$ has a boundary  fixed point at $\sigma_n$, and by Proposition \ref{bsem}, $\phi_t$ has a boundary  fixed point at $\sigma_n$ for all $t\geq 0$. But $\sigma_n$ is a regular contact point for $\phi_{t_0/2^n}$ hence in fact it is a boundary regular fixed point for $\phi_{t_0/2^n}$, which implies that $\sigma_n$ is a boundary regular fixed point for $\phi_t$ for all $t\geq 0$ by Proposition \ref{bsem}. Also,
\[
\phi_{t_0}(x)=\phi_{t_0/2^n}^{\circ 2^n}(x)=\phi_{t_0/2^n}^{\circ (2^n-1)}(\phi_{t_0/2^n}(x))=\phi_{t_0/2^n}^{\circ (2^n-1)}(\sigma_n)=\sigma_n.
\]
Now, by hypothesis $x$ is not a boundary regular fixed point for $\phi_{t_0}$, hence, $x\neq \sigma_n$. Thus
$\phi_{t_0}(x)=\phi_{t_0}(\sigma_n)$ and $\sigma_n$ has two distinct pre-images $x, \sigma_n$ by the map $\phi_{t_0}$.
By \cite[Lemma 8.2]{CP} it follows that either $\al_{\phi_{t_0}}(x)=+\infty$ or $\al_{\phi_{t_0}}(\sigma_n)=+\infty$, namely, both $x$ and $\sigma_n$ can not be regular contact points for $\phi_{t_0}$, reaching a contradiction. Thus we have proved that $x$ is a regular contact point for $\phi_{t_0/2^n}$ for all $n\in \N$ and  $\sigma_m\neq \sigma_n$ for $m\neq n$.

By \eqref{diffeq} we have for all $r\in (0,1)$
\[
|G(\phi_{t_0/2^n}(rx))||\sigma_n-\phi_{t_0/2^n}(rx)|=|\phi'_{t_0/2^n}(rx)| |G(rx)||1-r|\frac{|\sigma_n-\phi_{t_0/2^n}(rx)|}{|1-r|}.
\]
Taking the limit for $r\to 1$, using the Julia-Wolff-Carath\'eodory Theorem \ref{JWC} and taking into account that the curve $(0,1)\ni r\mapsto \phi_{t_0/2^n}(rx)$ is non-tangential because $x$ is a regular contact point (see, {\sl e.g.} \cite[pag. 54]{CMbook}), we obtain
\[
\lim_{r\to 1} |G(\phi_{t_0/2^n}(rx))||\sigma_n-\phi_{t_0/2^n}(rx)|=C (\al_{\phi_{t_0/2^n}}(x))^2\geq C,
\]
because $\al_{\phi_{t_0/2^n}(x)}>1$ by Remark \ref{magg}. Therefore for each $n\in \N$ we have that $\sigma_n\in  \cup_{D\geq C}\Pl_D(G)$. This implies that the set $\cup_{D\geq C}\Pl_D(G)$ is infinite, against Proposition \ref{finito}. We have finally reached a contradiction and claim \eqref{claim1} is proved.

Using Lemmas \ref{second-der}, \ref{noboundary} and equality \eqref{diffeq}, statements (1), (2), (3) or (6)   imply (a) and (b) and (4).

Finally, if (5) holds, the mean value theorem implies (6).
\end{proof}

In the previous proof  we showed the following fact: if $(\phi_t)$ is a semigroup of holomorphic self-maps of $\D$ such that $x\in \de \D$ is a regular contact point  for $\phi_{t_0}$ for some $t_0>0$ then  $x$ is a regular contact point for $\phi_{t_0/2^n}$ for all $n\in \N$. Moreover, if $x$ is not a boundary regular fixed point for $\phi_{t_0}$ and  $\sigma_n:=\angle \lim_{z\to x}\phi_{t_0/2^n}(z)$, then  $\sigma_n\neq \sigma_m$ for $n\neq m$. This fact is in a sense ``sharp'', as the following example shows:

\begin{example}
Let $\Omega:=\{w\in \C : \Re w>0\}\setminus (0,2]\times\{0\}$. Let $\v: \D \to \Omega$ be a Riemann mapping. Consider the semigroup $\phi_t:=\v^{-1}(\v(z)+t)$, $t\geq 0$. Let $a=(1, 0)$. Then there exists two points $x_1, x_2\in \de \D$ such that $\v(x_1)=\v(x_2)=a$. It is possible to prove that both $x_1, x_2$ are boundary regular contact point for $\phi_t$ for $t< 1$ but $|\phi_t(x_1)|, |\phi_t(x_2)|<1$ for $t>2$.
\end{example}

Now we are going to relate the regular poles of an infinitesimal generator to the $\beta$-points of the associated K\"onigs function:

\begin{theorem}\label{main2} Let $G\in \Gen$ and  let $h:\D \to \C$ be the associated normalized K\"onigs function. The following are equivalent:
\begin{enumerate}
\item $x\in \Pl(G)$,
  \item $x\in \NC(h)$,
  \item $\displaystyle{\liminf_{(0,1)\ni r\mapsto 1} \frac{|h'(rx)|}{1-r}<+\infty}$.
    \item $\angle\lim_{z\to x} h(z)\in \C$, $\angle\lim_{z\to x}h'(z)=0$ and  $\angle\lim_{z\to x}h''(z)=\ell\in \C$,
  \item $\lim_{(0,1)\ni r\mapsto 1} |h'(rx)|=0$, $\limsup_{(0,1)\ni r\mapsto 1} |h''(rx)|<+\infty$,
  \end{enumerate}
Moreover, let $\tau\in\oD$ be the Denjoy-Wolff point of $G$. If the previous conditions are satisfied and  $x\in \Pl_{|A|}(G)$ with $A=\angle\lim_{z\to x} G(z)(z-x)$, then
\begin{itemize}
  \item[a)] If $\tau \in \D$ then
  \[
  \angle\lim_{z\to x}h''(z)=\angle \lim_{z\to x}\frac{h'(z)}{z-x}=\frac{h(x)G'(\tau)}{A}\neq 0,
  \]
  where $h(x)=\angle\lim_{z\to x}h(z)$.
  \item[b)]  If $\tau \in \de\D$  then
  \[
  \angle\lim_{z\to x}h''(z)=\angle \lim_{z\to x}\frac{h'(z)}{z-x}=\frac{1}{A}.
  \]
\end{itemize}
\end{theorem}

\begin{proof}
Clearly, (2) implies (3).

In case $\tau\in \de \D$, then $G(z)=1/h'(z)$, hence
\begin{equation}\label{h-Gfrontiera}
\frac{h'(z)}{z-x}=\frac{1}{G(z)(z-x)}.
\end{equation}
From here and from Lemma \ref{exists} we have clearly that (1) is equivalent to (2) and (3).

In case $\tau\in \D$ we have $G(z)=G'(\tau)h(z)/h'(z)$ and $h(\tau)=0$. Therefore
\begin{equation}\label{h-Ginterior}
\frac{h'(z)}{z-x}\frac{1}{h(z)}=\frac{G'(\tau)}{G(z)(z-x)}.
\end{equation}
Assume (1) holds and let $(\phi_t)_{t\geq 0}$ be the semigroup of holomorphic self-maps of $\D$ generated by $G$. Since $h(\phi_t(z))=e^{G'(\tau)t}h(z)$, from Theorem \ref{main}.a)  it follows that
\[
\angle \lim_{z\to x} h(z)=e^{-G'(\tau)t}h(\sigma_t)\in \C.
\]
Hence (2) follows from \eqref{h-Ginterior}.
Conversely, if (3) holds, since $\angle\liminf_{z\to x}|h(z)|>0$ (because $h$ is univalent and $h(\tau)=0$), \eqref{h-Ginterior} and Lemma \ref{exists} imply (1).

Now, if (1)---(3) are satisfied, then (a) and (b) follow from \eqref{h-Ginterior}, \eqref{h-Gfrontiera} and Lemma \ref{second-der}. Hence, (4) and (5) hold.

Clearly, (4) implies (5). Finally, if (5) holds, the mean value theorem applied to the real and imaginary part of $h'$, implies (3).
\end{proof}

\begin{remark}
Let $h:\D \to \C$ be univalent and star-like, $h(0)=0, h'(0)=1$, and let $\Omega:=h(\D)$. We can define the semigroup of holomorphic self-maps of $\D$ whose elements are $\phi_t:=h^{-1}(e^{-t}h(z))$. A direct computation shows that the associate infinitesimal generator is of the form $G(z)=-zp(z)$, with $\Re p(z)>0$ for all $z\in \D$ and $p(0)=1$. From \cite[Theorem 2.1]{Be} and Theorem \ref{main2} we infer that the points of $\de \Omega$ with positive $\beta$-numbers correspond to regular poles of $G$ and, if $A>0$ and if $x\in \Pl_A(G)$  then the $\beta$-number of $h(x)$ with respect to $0$ is given by
\[
\beta_\Omega(0, h(x))=2|h(x)|A.
\]
\end{remark}

\section{Dual generators}\label{5}

\begin{definition}
Let $G\in \Gen$, $G\not\equiv 0$. Let $G(z)=(\tau-z)(1-\overline{\tau}z)p(z)$ be the Berkson-Porta decomposition of $G$, where $\tau\in \oD$ is the Denjoy-Wolff point of the associated semigroup and $p:\D \to \C$ holomorphic with $\Re p(z)\geq 0$ for all $z\in \D$. We define the {\sl dual infinitesimal generator}
\[
\hat{G}(z):=(\tau-z)(1-\overline{\tau}z)\frac{1}{p(z)}.
\]
\end{definition}

\begin{remark}
Let  $G\in \Gen$, $G\not\equiv 0$ and let $\tau\in \oD$ be the Denjoy-Wolff point of the associated semigroup. Then $\hat{G}\in \Gen$ and the associated semigroup has Denjoy-Wolff point $\tau$. Moreover, $\hat{\hat{G}}=G$.

In fact, this follows at once from the Berkson-Porta formula as soon as one realizes that $1/p:\D \to \C$ is holomorphic (because $p(z)=0$ for some $z\in \D$ if and only if $p\equiv 0$ if and only if $G\equiv 0$) and $\Re (1/p(z))\geq 0$.
\end{remark}

As a matter of notation, if $G\not\equiv 0$ is an infinitesimal generator in $\D$,  we denote by $(\hat{\phi}_t)$ the semigroup and by $\hat{h}$ the K\"onigs function associated with $\hat{G}$.

Notice that
\begin{equation}\label{GGhat}
G(z)\hat{G}(z)=(\tau-z)^2(1-\overline{\tau}z)^2.
\end{equation}

From this formula we immediately have:

\begin{proposition}\label{BRNP-P}
Let $G\in \Gen$, $G\not\equiv 0$ with Denjoy-Wolff $\tau\in \oD$. Let $x\in \de \D\setminus\{\tau\}$. Then the following are equivalent:
\begin{enumerate}
  \item $x$ is a boundary regular null point  for $G$ of dilation $\ell:=\lim_{r\to 1}|G(rx)|/(1-r)>0$
  \item $x$ is a regular pole  for $\hat{G}$ of mass $|\tau-x|^4/\ell$.
\end{enumerate}
\end{proposition}

Now we give the following complete characterization of boundary regular null points:

\begin{proposition}
Let $G\in \Gen$, $G\not\equiv 0$, with Denjoy-Wolff point $\tau\in\oD$. Let $h$ be the associated K\"onigs function and $(\phi_t)$ the associated semigroup. Let $x\in \de\D\setminus\{\tau\}$. The following are equivalent:
\begin{enumerate}
\item $\liminf_{(0,1)\ni r\to 1} |G(rx)|/(1-r)<+\infty$,
  \item $x$ is a boundary regular null point for $G$,
  \item $x$ is a boundary regular fixed point for some---and hence any---$\phi_t$, $t>0$,
\item there exists $\rho>0$ such that \begin{itemize}
  \item  $\displaystyle{\lim_{(0,1)\ni r\to 1}\frac{\log |h(rx)|}{\Re G^{\prime }(\tau )\log  (1-r)}=\rho}$ if $\tau\in \D$,
  \item $\displaystyle{\lim_{(0,1)\ni r\to 1}\frac{\Re h(rx)}{\log(1-r)} =\rho}$  if $\tau\in \de \D$,
  \end{itemize}
  \item \begin{itemize}
  \item  $\displaystyle{\limsup_{(0,1)\ni r\to 1} \frac{|h'(rx)|(1-r)}{|h(rx)|}>0}$ if $\tau\in \D$,
  \item $\displaystyle{\limsup_{(0,1)\ni r\to 1} |h'(rx)|(1-r)>0}$ if $\tau\in \de \D$,
  \end{itemize}
\end{enumerate}
Moreover, if the previous conditions are satisfied and $\ell:=\lim_{r\to 1}|G(rx)|/(1-r)>0$ then
\[
\angle\lim_{z\to x}|h(z)|=\infty,
\]
$\rho^{-1}=\ell$ and
 \begin{itemize}
  \item[a1)]  $\displaystyle{\angle\lim_{z\to x} \frac{h'(z)(z-x)}{h(z)}=\frac{G'(\tau)}{\ell}}$ if $\tau\in \D$,
  \item[a2)] $\displaystyle{\angle\lim_{z\to x} h'(z)(z-x)=\frac{1}{\ell}}$ if $\tau\in \de \D$.
  \end{itemize}
\end{proposition}

\begin{proof}
Let $\hat{G}$ be the dual infinitesimal generator associated with $G$. If (1) is satisfied, by \eqref{GGhat} and Theorem \ref{main}, it follows that $x$ is a regular pole for $\hat{G}$, hence $x$ is a boundary regular null point for $G$ by Proposition \ref{BRNP-P}, and (2) holds.

By Proposition \ref{bsem}, (2) and (3) are equivalent.
Next, from Proposition \ref{Unival-VectorField},   (2) and (5) are equivalent.

In order to prove that (2) is equivalent to (4), let define $g(r):=\Re \log (h(rx)/(rx-\tau))$ in case $\tau \in \D$ and $g(r):=\Re h(rx)$ in case $\tau\in \de\D$, for $r\in (0,1)$. Thanks to Lemma \ref{Co-Po} we can use l'H\^opital's rule, and by  Proposition \ref{Unival-VectorField} it follows easily that
\[
\lim_{r\to 1} \frac{g(r)}{-\log (1-r)}=\lim_{r\to 1} g'(r)(1-r)=\begin{cases}
-\lim_{r\to 1}\frac{\Re G'(\tau)(xr-x)}{G(rx)} \quad \hbox{if}\ \tau\in  \D\\
-\lim_{r\to 1}\frac{xr-x}{G(rx)} \quad \hbox{if}\ \tau\in \de \D
\end{cases}
\]
from which the equivalence between (2) and (4).

Finally, if (2) holds, then a1) and a2) follow from a direct computation from Proposition~\ref{Unival-VectorField}.
\end{proof}

\begin{remark} We point out that the existence of the angular limit of $h$ at a boundary regular  fixed point was previously known. If $\tau\in \partial \D$, this was proved in \cite[page 268]{CD}.   If $\tau\in \D$, in \cite[Theorem 2]{CDP} it has been proved that actually $\lim_{z\to x} h(z)=\infty$ for any boundary fixed (non necessarily regular) point of the semigroup.
 \end{remark}

\section{Geometry near $\beta$-points of K\"onigs functions}\label{beta}

Let $h$ be the K\"onigs function of a semigroup $(\phi_{t})$ of holomorphic self-maps of $\D$ generated by the infinitesimal generator $G$. Let $x\in \de\D$ be a $\beta$-point of $(\phi_t)$, which, as we saw before, corresponds to a regular pole of $G$ and to a $\beta$-point of $h$. We want to understand how the image $h(\D)$ looks like near $h(x)$, or, better, which local geometry of $h(x)$ implies that $x$ is a regular pole.

We start with the following measure-theoretic consideration. Let $G(z)=(z-\tau)(1-\overline{\tau}z)p(z)$ be the Berkson-Porta decomposition of $G$. Then $p:\D \to \C$ is holomorphic and $\Re p(z)\geq 0$ for all $z\in \D$. Moreover, since we assumed that $x\in\de\D$ is a regular pole, it follows that in fact $p(z)>0$ for all $z\in \D$. Hence, $\Re (1/p(z))>0$ for all $z\in \D$. Therefore there exists a positive finite Borel measure $\mu$ on $\de \D$ such that
\begin{equation}\label{represent-p}
\frac{1}{p(z)}=\frac{1}{2\pi}\int_0^{2\pi}\frac{e^{it}+z}{e^{it}-z}d\mu(t).
\end{equation}

\begin{proposition}\label{mu-nc}
Let $G(z)=(z-\tau)(1-\overline{\tau}z)p(z)$ and let $x\in \de\D$. Let $\mu$ be the positive finite Borel measure on $\de\D$ which satisfies \eqref{represent-p}. Then $x=e^{i\theta}\in\Pl(G)$ if and only if
\begin{equation}\label{formulamu}
\begin{cases}
    \displaystyle{\lim_{(0,1)\ni r\to 1} \int_0^{2\pi}\frac{d\mu(t)}{|e^{it}-re^{i\theta}|^2}<+\infty,}  \\
     \displaystyle{\lim_{(0,1)\ni r\to 1}\int_0^{2\pi} \frac{\sin(t-\theta)}{|e^{it}-re^{i\theta}|^2}\frac{d\mu(t)}{1-r}=0.  }
  \end{cases}
\end{equation}
\end{proposition}

\begin{proof}
By definition of regular pole,  $\lim_{r\to 1}(1-r)|G(rx)|>0$. Hence, taking into account the Berkson-Porta decomposition of $G$, this is equivalent to $\lim_{r\to 1}(1-r)|p(rx)|>0$. By Lemma \ref{Co-Po} the last limit always exists finite and non-negative, thus   by (\ref{represent-p}),  $x$ is a regular pole of $G$ if and only if
$$
\lim_{r\to1}\frac{1}{1-r}\Re\left(\frac{1}{p(rx)}\right)=\frac{1}{\pi}\lim_{r\to1} \int_0^{2\pi}\frac{d\mu(t)}{|e^{it}-re^{i\theta}|^2}<+\infty,
$$
and
$$
\lim_{r\to1}\frac{1}{1-r}\Im\left(\frac{1}{p(rx)}\right)=\frac{1}{\pi}\lim_{r\to1}\frac{1}{1-r} \ \int_0^{2\pi}\frac{\sin(t-\theta)}{|e^{it}-re^{i\theta}|^2}\, d\mu(t)<+\infty.
$$
Finally, since  $\lim_{r\to 1}\frac{1}{(1-r)p(rx)}\in (0,+\infty]$, we have that $\lim_{r\to1}\frac{1}{1-r}\Im\left(\frac{1}{p(rx)}\right)<+\infty$ if and only if $\lim_{r\to1}\frac{1}{1-r}\Im\left(\frac{1}{p(rx)}\right)=0$.
\end{proof}

\begin{remark}
By Fatou's Lemma, we know that
$$
\int_0^{2\pi}\frac{d\mu(t)}{|e^{it}-e^{i\theta}|^2}\leq \lim_{(0,1)\ni r\to 1} \int_0^{2\pi}\frac{d\mu(t)}{|e^{it}-re^{i\theta}|^2}.
$$
Then, if the function $[0,2\pi]\ni t\mapsto \frac{1}{|e^{it}-e^{i\theta}|}$ does not belong to $L^2(\mu)$, the point $e^{i\theta}$ cannot be a regular pole of $G$.
\end{remark}

\subsection{Star-like functions and radial slits} Assume $0$ is the Denjoy-Wolff point of the semigroup and that $h$ is star-like with respect to the origin. This condition, due to Proposition \ref{Unival-VectorField}.(1), is equivalent to $p(0)>0$ and to $G'(0)<0$. Moreover, in such a case by Herglotz' representation formula we have
\begin{equation}\label{h-mu}
\frac{zh'(z)}{h(z)}=\frac{p(0)}{2\pi}\int_0^{2\pi}\frac{e^{it}+z}{e^{it}-z}d\mu(t),
\end{equation}
where $\mu$ is the positive finite Borel measure on $\de \D$ defined in \eqref{represent-p}.

Recall that any positive finite Borel measure on $\de \D$ is defined uniquely by a non-decreasing real function $\upsilon:[0,2\pi]\to \R$ such that $\upsilon(0)=0$ and $\upsilon(2\pi)=2\pi$.

According to \cite[Theorem 3.18 and formula (10) pag. 67]{P2}, the previous measure $\mu(t)$ is associated with the function
\[
\upsilon(t):=\lim_{r\to 1}\arg h(re^{it}),
\]
(the limit exists for all $e^{it}$).

A {\sl radial segment} is a set in $\C$ of the form $\{z\in \C: z=e^s z_0, s\in [a,b)\}$ with $a,b\in [-\infty, +\infty]$, $a<b$, and $z_0\in \C\setminus\{0\}$. The point $e^az_0$ is called the tip of the radial segment. We say that $h$ has a {\sl radial segment slit} if  $\de (h(\D))$ contains a radial segment such that any local arcwise connected component containing the tip remains arcwise connected when the tip is removed. A radial segment slit $R$ is ``isolated'' in $\de (h(\D))$ if $\de (h(\D))$ is locally arcwise connected at each point $p\in R$. By Carath\'eodory's theorem (see, {\sl e.g.}, \cite{P2}), if $h$ has a radial segment slit $R$ which is isolated in $\de (h(\D))$ (for short an ``isolated radial segment slit''), there exists an open connected non-empty arc $A\subset \de\D$ of positive Lebesgue measure such that $h$ extends continuously to $A$, $h(A)=R$ and a point $x$ in the interior of $A$ such that $h(x)$ is the tip of $R$. By the previous considerations,  $\upsilon(t)=\hbox{constant}$ for all $t\in A$, which, in turn, is equivalent to $\mu(A)=0$.

\begin{proposition}\label{plinio}
Let $G\in\Gen$, with associated K\"onigs function $h:\D\to \C$. Assume that $G(0)=0$ and $G'(0)<0$. Suppose $h$ has an isolated radial segment slit  whose tip is $h(e^{i\theta_0})$, for some $e^{i\theta_0}\in \de\D$. Then $e^{i\theta_0}\in \Pl(G)$.
\end{proposition}

\begin{proof}
By Theorem \ref{main2}, $x$ is a regular pole of $G$ if and only if it is a $\beta$-point for $h$, and it is well known (see, {\sl e.g.}, \cite{Be}) that the tip of an isolated radial segment slit has positive $\beta$-number, which by Bertilsson's \cite[Theorem 2.1]{Be}, means that it is the image of a $\beta$-point of $h$. However, we provide here a different proof of this fact based on Proposition~\ref{mu-nc}.

Thus, suppose $R$ is an isolated radial segment slit for $h$, corresponding to the open connected arc $A\subset \de\D$ of positive Lebesgue measure. Since $\mu(A)=0$, for any $e^{i\theta}\in A$, the first equation in \eqref{formulamu} is satisfied.

Now, for $e^{i\theta}\in A$, let
\[
k(\theta,r):=\int_0^{2\pi} \frac{\sin (t-\theta)}{|e^{it}-re^{i\theta}|^2}d\mu(t).
\]
Using the Lebesgue dominated convergence theorem, one can prove that   $\frac{\de k(\theta,1)}{\de \theta}>0$ (recall that $A$ is open in $\de \D$). Thus, $k(\theta,1)$ has at most one zero in  $A$. It is well known (see also Lemma \ref{local} below for a general argument) that, if $e^{i\theta_0}$ is a tip of a radial segment, then $\lim_{r\to 1}h'(re^{i\theta_0})=0$. Therefore, from \eqref{h-mu} we have $k(\theta_0,1)=0$.

Now, another computation, using again the Lebesgue dominated converge theorem, shows that $\frac{\de k(\theta,1)}{\de r}>0$, hence, by the mean value theorem, it follows that
\[
k(\theta_0,r)/(1-r)\leq C<+\infty.
\]
 Since $\lim_{r\to 1}(1-r)p(rx)\in (0,+\infty]$ by Lemma \ref{Co-Po}, from \eqref{represent-p} and \eqref{h-mu}, the previous condition is equivalent to  $k(\theta_0,r)/(1-r)\to 0$, and by Proposition \ref{mu-nc} we are done.
\end{proof}

\begin{remark}
Note that, as it is clear from geometric considerations, any point of an isolated radial segment slit but the tip is not a regular pole for the infinitesimal generator.
\end{remark}

\subsection{Other examples} In some cases, one can ``localize'' the problem of understanding which points are  $\beta$-points by means of the following trick:

\begin{lemma}\label{local}
Let $h:\D \to \C$ be univalent. Let $A\subset \de \D$ be a connected arc of positive Lebesgue measure. Assume that $h$ extends continuously on $A$. Let $g:\D \to \C$ be another univalent map  such that $g(\D)\subset h(\D)$, and assume there exists $A'\subseteq A$ a connected arc of positive Lebesgue measure such that $g$  extends continuously on $A'$ and $g(A')=h(A')$. Let $x$ be a point in the interior of $A'$  such that $g(x)=h(x)$. Then $x\in \NC(h)$ if and only if $x\in \NC(g)$.
\end{lemma}

\begin{proof}
The function $\v:=h^{-1}\circ g: \D \to \D$ is univalent. Moreover, by hypothesis, using Lindel\"of's theorem, it is not difficult to see that $\v$ extends continuously on $A'$ and $\v(A')\subset \de \D$. Then, by Schwarz' reflection principle,  $\v$ extends holomorphically through $A'$. In particular, $\v(x)=x$ and $\v'(x)$ exists finite. By Julia's lemma, $\v'(x)>0$ and $(0,1)\ni r\mapsto \v(rx)$ is a non-tangential curve approaching $x$. Hence for $r\in (0,1)$ we have
\[
\frac{|h'(\v(rx))|}{|x-\v(rx)|}\frac{|x-\v(rx)|}{1-r}|\v'(rx)|=\frac{|g'(rx)|}{1-r}
\]
and thus, taking the limsup as $r\to 1$, it is clear that $x\in \NC(h)$ if and only if $x\in \NC(g)$.
\end{proof}

Now we present a couple of examples of images which produce regular poles:

\begin{example}
Let $\upsilon(t):=0$ for $0\leq t\leq \pi$ and $\upsilon(t):=2(t-\pi)$ for $\pi\leq t\leq 2\pi$. Let $\mu$ be the positive finite Borel measure on $\de\D$ defined by $\upsilon$. And let $h:\D\to \C$ be the function star-like with respect to $0$   which satisfies \eqref{h-mu} with $h'(0)=1$. As $\upsilon$ represents the argument of (the non-tangential limit of) $h(e^{it})$ for $t\in [0,2\pi]$, we see that $h(\D)$ is a Jordan domain minus a radial segment slit, whose tip is $h(i)$. By Proposition \ref{plinio}, we know that $i$ is a regular pole (in fact it is the unique regular pole) for the semigroup $(h^{-1}(e^{-t} h(z)))_{t\geq 0}$. It is however interesting to see the actual computations in such a simple case using directly Proposition \ref{mu-nc}:
$$
\int_0^{2\pi}\frac{d\upsilon(t)}{|e^{it}-re^{i\pi/2}|^2}=2\int_\pi^{2\pi}\frac{dt}{|e^{it}-re^{i\pi/2}|^2} \leq \frac{2}{1+r^2}\pi
$$
and
$$
\int_0^{2\pi} \frac{\sin(t-\pi/2)}{|e^{it}-re^{i\theta}|^2}d\upsilon(t)=2\int_\pi^{2\pi} \frac{\sin(t-\pi/2)}{|e^{it}-re^{i\theta}|^2}dt=0.
$$
\end{example}

\begin{example}
Let $\upsilon(t):=-\pi^{1-\al}(-t+\pi)^\al$ for $0\leq t\leq \pi$ and $\upsilon(t):=\pi^{1-\al}(t-\pi)^\al$  for $\pi\leq t\leq 2\pi$, with $\al>0$.  Let $\mu$ be the positive finite Borel measure on $\de\D$ defined by $\upsilon$. Let $G(z)=-zp(z)$ where $p(z)$ is defined by \eqref{represent-p}. A direct computation shows that $\mu$ satisfies the hypothesis of  Proposition \ref{mu-nc} at $\pi$ if and only if $\al>2$. Thus $\upsilon(\pi)=1\in \Pl(G)$ if and only if $\al>2$. Moreover, by Fatou's lemma, since $\upsilon$ is differentiable at every point but at most $t=\pi$, it follows that $\angle\lim_{z\to e^{it}}1/p(z)=\upsilon'(t)\neq 0$ for all $t\in [0,2\pi]\setminus\{\pi\}$. Thus $e^{i\theta}$ is not a boundary regular null point for $G$ for all $\theta\in (0,2\pi)$.

If $h:\D \to \C$ is the K\"onigs function associated with $G$, the image $h(\D)$ is, for $0<\al<1$ a disc with a ``cusp exiting from $1$'' and, for $\al>1$ a disc minus a ``cusp entering from $1$'' (see Figure 1) 
\begin{figure}
\centering
\fbox{
\includegraphics[height=3.0cm]{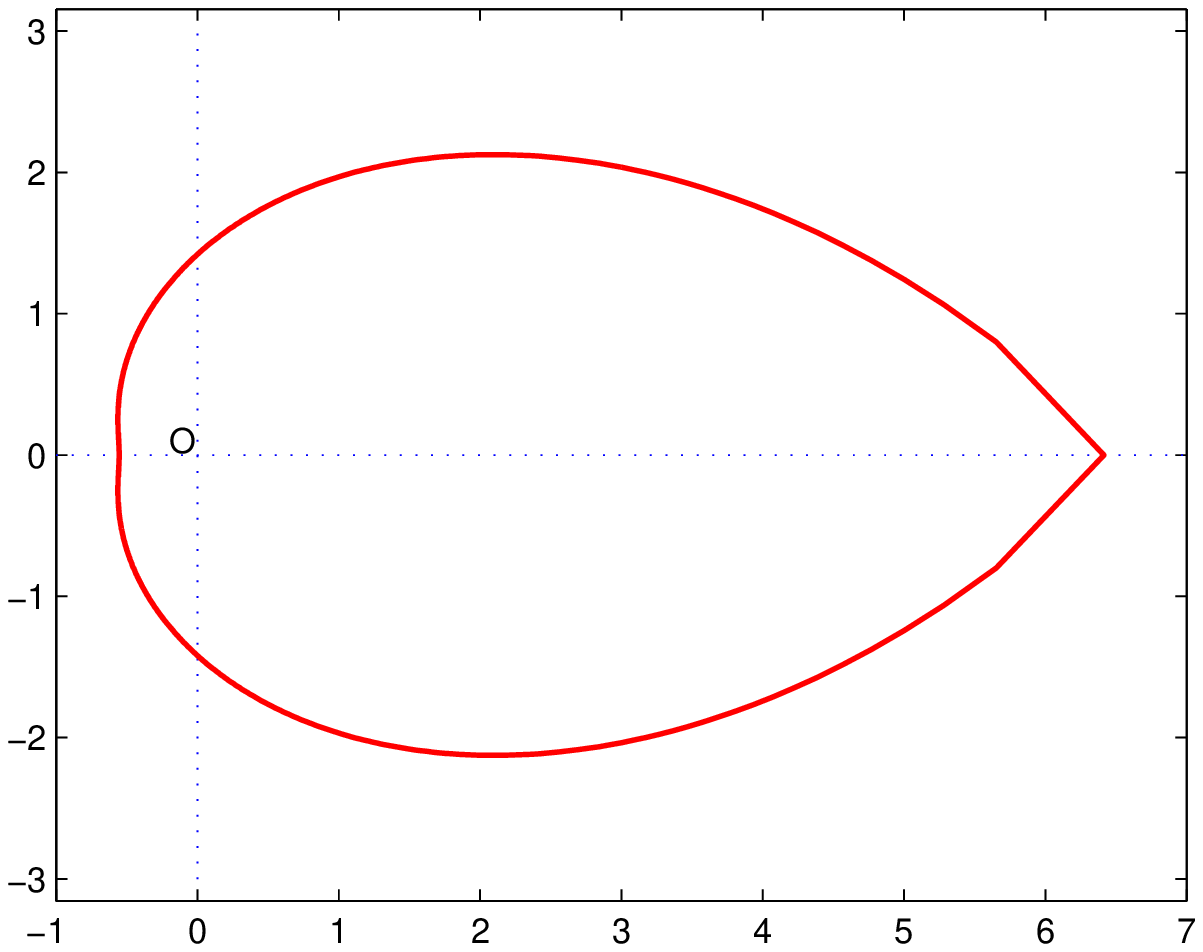}}
\hspace{2cm}
\fbox{
\includegraphics[height=3.0cm]{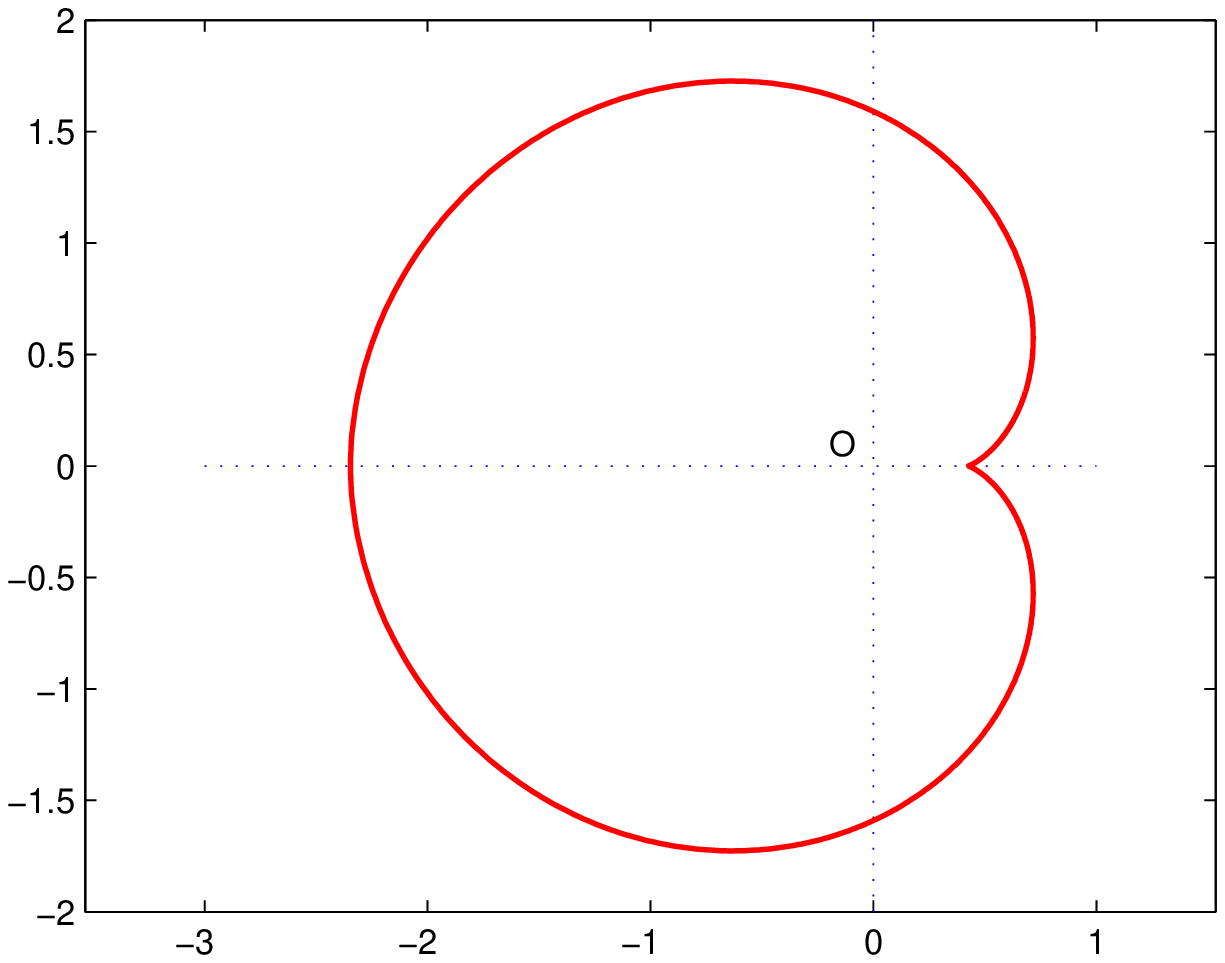}
}
\caption{$\al=1/2$ and $\alpha=2$}
\end{figure}
\end{example}

The previous examples and a direct application of Lemma \ref{local} allow to prove the following result:

\begin{proposition}
Let $G\in \Gen$. Let $h:\D \to \C$ be the K\"onigs function associated with $G$. Assume that $\de(h(\D))$ contains a curve $\Gamma$ which can be parameterized by a continuous function $\gamma:(-1,1)\to \Gamma$ such that, up to rigid movements of $\C$,
\begin{enumerate}
  \item either  $\arg \gamma(t)=\hbox{constant}$ (that is $\Gamma$ is a linear slit)
  \item or,  $\arg \gamma(t)=-\pi^{1-\al}(-t)^\al$ for $-1< t\leq 0$ and $\arg \gamma(t):=\pi^{1-\al}t^\al$ for $0< t< 1$, with some $\al>2$ (that is $\Gamma$ is a cusp slit with vertex at $\gamma(0)$) and $h(\D)$ contains an angle of (any) positive amplitude at $\gamma(0)$.
\end{enumerate}
If $\Gamma$ is locally arcwise connected in $\C\setminus h(\D)$ then $\gamma(0)$ is a regular pole of $G$.
\end{proposition}

\section{Radial multi-slits semigroups}\label{ultima}

Let $\Gamma$ be a {\sl radial slit}, namely $\Gamma=\{z\in \C: z=r T, r\in [s,+\infty)\}$, where $T\in \C\setminus \{0\}$ and $s>0$. A {\sl radial $m$-slit domains} $\Omega$ is given by $\C\setminus \cup_{j=1}^m \Gamma_j$ where the $\Gamma_j$'s are (different) radial slits. As $\Omega$ is star-like with respect to $0$, there exists a unique Riemann map $h:\D \to \Omega$ such that $h(0)=0$, $h'(0)>0$. Up to dilation, we can, and we will, assume that $h'(0)=1$. Since $\de\Omega$ is locally arcwise connected, $h$ extends continuously (as a map with values in $\C\mathbb P^1$) up to $\de\D$.

We call $T_1$ one of the tips of the $m$ radial slits with minimal real part, and label the other tips $T_j$, $j=2,\ldots, m$,  in such a way that $T_j$ follows $T_{j-1}$ clockwise. We let $a_1,\ldots,a_m\in \de \D$ be such that $h(a_j)=T_j$ (as the $T_j$'s are not cut-points, there is only one pre-image of each $T_j$). Let $b_1,\ldots, b_m\in \de \D$ denote the $m$ pre-images of $\infty$ under $h$ ($\infty$ is a cut-point of multiplicity $m$, thus there are exactly $m$ preimages). It is clear that in the arc $[a_j,a_{j+1}]$ (with $j$ counted mod $m$) there is only one of such points, and we label them in such a way that $b_j\in [a_j,a_{j+1}]$ for $j=1,\ldots, m-1$ and $b_m\in [a_{m}, a_1]$. Also, let $2\pi \sigma_j$, with $\sigma_j>0$, $j=1,\ldots, m-1$ be  the amplitude of the angle formed by vectors $T_j-0$ and $T_{j+1}-0$ for $j=1,\ldots, m-1$ and let $2\pi\sigma_m$ be the amplitude of the angle formed by the vectors $T_m-0$ and $T_{1}-0$. By definition, $\sum_{j=1}^m \sigma_j=1$.

Consider the semigroup of holomorphic self-maps of $\D$ defined by $\phi_t:=h^{-1}(e^{-t}h(z))$ and call it a {\sl radial $m$-slits semigroup}. Clearly, $0$ is its Denjoy-Wolff point.

\begin{proposition}\label{forma-pradial}
Let $(\phi_t)$ be a radial $m$-slits semigroup, and let $G(z)=-zp(z)$ be the associated infinitesimal generator. Then $G$ has $m$ boundary regular null points $b_1,\ldots, b_m\in\de\D$ which  dilations, respectively, $2\sigma_1,\ldots, 2\sigma_m$. Also, $G$ has $m$ regular poles $a_1,\ldots, a_m$ with mass, respectively, $2\mu_1,\ldots, 2\mu_m>0$ such that $\sum_{j=1}^m \mu_j=1$. Moreover,
\begin{equation}\label{p-1sup}
p(z)=\sum_{j=1}^m \mu_j \frac{a_j+z}{a_j-z}, \quad \frac{1}{p(z)}=\sum_{j=1}^m \sigma_j \frac{b_j+z}{b_j-z}.
\end{equation}
Conversely, if $p(z)$ is given by \eqref{p-1sup}, then the infinitesimal generator $G(z)=-zp(z)$ generates a radial $m$-slits semigroup.
\end{proposition}

\begin{proof}
As recalled before, by  \cite[Theorem 3.18 and formula (10) pag. 67]{P2}, the measure $\mu(t)$ which represents $1/p$ in \eqref{represent-p} (and  given by \eqref{h-mu}) is $\mu(t)=d\upsilon(t)$ where $\upsilon(t)=\lim_{r\to 1}\arg h(re^{it})$. Hence, it follows easily that $\mu(t)=\sum_{j=1}^m \sigma_j \delta_{b_j}$ where $\delta_{b_j}$ is the Dirac measure concentrated in $b_j$ of mass $1$, from which \eqref{p-1sup} for $1/p(z)$.

Now we want to show that $p(z)$ has the claimed form. By Proposition \ref{plinio}, the points $a_1,\ldots, a_m$ are regular poles of $G$, with some mass $\mu_1,\ldots, \mu_m>0$. Since $1/p(z)=-1/\overline{p(1/\overline{z})}$ for all $z\in \C$, it follows that these are the only poles of $p$. Hence the map $q(z)=p(z)-\sum_{j=1}^m \mu_j (a_j+z)/(a_j-z)$ is a bounded rational map of $\C \mathbb P^1$, thus it is constant. Now, by Lemma \ref{Co-Po}, $\Re q(z)\geq 0$ for all $z\in \D$, hence, $q(z)\equiv r\geq 0$. Now, for $z=b_1$ (one of the zeros of $p$) we have
\[
r=p(b_1)-\sum_{j=1}^m \mu_j (a_j+b_1)/(a_j-b_1)=-\sum_{j=1}^m \mu_j (a_j+b_1)/(a_j-b_1)\in i\R,
\]
from which it follows that $r=0$, and $p$ has the desired form.

To end up the proof, assume that $p$ satisfies \eqref{p-1sup}. Then let $q(z)=1/p(z)$. By  \eqref{represent-p},  it follows that the measure $\mu(t)$ associated with $q(z)$ has atoms at $a_1,\ldots, a_m$ with mass $\mu_1,\ldots, \mu_m$ respectively. Let $g:\D \to \C$ be the K\"onigs map associated with $-zq(z)$ such that $g(0)=0, g'(0)=1$. Hence, the function $\upsilon(t):=\lim_{r\to 1} g(re^{it})$ by \eqref{h-mu} and \cite[Theorem 3.18 and formula (10) pag. 67]{P2}, is constant on the arc $(a_j,a_{j+1})$ ($j=1\ldots, m$ mod $m$), therefore the associated domain is a radial multi-slits domain. Applying what we already proved, we obtain that $q(z)$ is of the form given in \eqref{p-1sup}, hence, $1/p(z)=q(z)$ is of the same type and, repeating the above argument, we obtain that $-zp(z)$ generates a radial multi-slit domain.
\end{proof}

\begin{remark}
The form of the generating (that also, in general, might be time-dependent) holomorphic vector field  $-zp(z)$ in case of multi-slits can also be deduced using Loewner's differential equation (see, {\sl e.g.}, \cite{Pe}). The simpler proof we gave here in the star-like case seems to be new, and the result gives a new insight on the geometrical meaning of the terms appearing. We thank Pavel Gumenyuk for helping us to simplify the original proof we found.
\end{remark}

In the following example we show that the tip of a non-isolated radial slit might not correspond to a regular pole:
\begin{example}\label{no-tip}
Let $\{\theta_j\}\subset (0,1/2)$ be a sequence monotonically decreasing to $0$. Let $T_\infty:=1$ and let $T_j:=e^{2\pi i \theta_j}$ for $j\in \N$. Let $\Gamma_j:=\{sT_j: s\geq 1\}$ for $j\in \N\cup \{\infty\}$ and set
\[
\Omega_m:=\C\setminus \left(\bigcup_{k=1}^m \Gamma_k \cup \overline{\Gamma_k}\cup \Gamma_\infty\right), \quad m\in \N\cup\{\infty\}.
\]
For a fixed $m\in \N$, the domain $\Omega_m$ is a radial $(2m+1)$-slit domain symmetric with respect to the real axis, and, as $m\to \infty$, the sequence $\{\Omega_m\}$ converges in the kernel sense  to the simply connected domain $\Omega_\infty$, which has  infinitely many isolated radial slits collapsing to a non-isolated one, $\Gamma_\infty$.

Fix $m\in \N \cup\{\infty\}$. Let $h_m:\D \to \Omega_m$ be the unique Riemann mapping normalized so that $h_m(0)=0, h_m'(0)>0$. Since $\Omega_m$ is symmetric with respect to the real axis, it follows that $h_m(\overline{z})=\overline{h_m(z)}$ for all $z\in \D$.  From this, it is not difficult to see  that $\lim_{(0,1)\ni r\to 1}h_m(r)=T_\infty$ for all $m\in \N\cup\{\infty\}$. By the prime-ends theory, $h_m$ is continuous (as a map with values in $\C\mathbb P^1$) up to $\de \D$ for $m\in \N$ while $h_\infty$ is continuous on $\oD\setminus\{1\}$ and has non-tangential limit $T_\infty$ at $1$. Again by symmetry,  if $a_{j,m}\in \de \D$ with $j=1,\ldots, m$ are such that $h_m(a_{j,m})=T_j$, then $\Im a_{j,m}>0$ and $h_m(\overline{a_{j,m}})=\overline{T_j}$.

Moreover, $\Re a_{j,m}>\Re a_{k,m}$ for all $k<j\in \{1,\ldots, m\}$. In order to see this, let $L_{j,1}$ be the segment between $1$ and $T_j$. Then $h_m^{-1}(L_{j,1})$ is a crosscut in $\D$ ending in $1$ and $a_{j,m}$ which divides $\D$ in two regions $A_0$ and $A_1$, say, with $0\in A_0$. Since $h_m(0)=0$, $h_m(A_1)$ is the region in $\Omega_m$ contained in the open convex hull between $\Gamma_\infty$ and $\Gamma_j$. As this region does not contain $T_k$, it follows that $a_{k,m}\in \overline{A_0}$, that is $\Re a_{k,m}<\Re a_{j,m}$. Also (see, {\sl e.g.}, \cite[Ex. 4 pag. 90]{P2})
\begin{equation}\label{veloc-a}
|1-a_{j,m}|^2\leq K_m |1-T_j|, \quad j=1,\ldots, m,
\end{equation}
for a certain constant $K_m>0$.

Note that the angle between $T_j-0$ and $T_{j+1}-0$ (clockwise orientation) is $2\pi (\theta_j-\theta_{j+1})$ for $j\in \N$, while the angle between $\overline{T_1}-0$ and $T_1-0$   is $2\pi (1-2\theta_1)$. Let denote by $2\pi\theta_{\infty,m}$ the angle between $T_m-0$ and $\overline{T_m}-0$.

For each fixed $m\in \N\cup \{\infty\}$, let $\phi^m_t(z):=h_m^{-1}(e^{-t}h_m(z))$. Then $(\phi_t^m)$ is a semigroup of holomorphic self-maps of $\D$ and we denote by $G_m(z):=-zp_m(z)$  the associated infinitesimal generator. By property (1).(iii) in Proposition \ref{Unival-VectorField} it follows that the K\"onigs function associated with $G_m(z)$ is $h_m(z)/h'_m(0)$. Note that this implies that $p_m(\overline{z})=\overline{p_m(z)}$ for all $z\in \D$.

If $m\neq\infty$, by Proposition \ref{forma-pradial}, the infinitesimal generator $G_m$ has $2m+1$ boundary regular null points. By symmetry, it is easy to see that $-1$ is one of such points. Moreover, if we label  by $b_{j,m}$   the boundary regular null point of $G_m$ contained in the arc between $a_{j,m}$ and $a_{j+1,m}$ with $j=1,\ldots, m-1$ and by $b_{m,m}$ the boundary regular null point contained in the arc between $a_{m,m}$ and $1$, then $\overline{b_{j,m}}$ for $j=1,\ldots, m$ are the other boundary regular null points of $G_m$. Hence, using the notations previously introduced, Proposition \ref{forma-pradial} implies that
\begin{equation}\label{p-m}
\begin{split}
\frac{1}{p_m(z)}&=\frac{(1-2\theta_1)}{2}\frac{1-z}{1+z}\\&+\sum_{j=1}^{m-1} (\theta_j-\theta_{j+1})\frac{1-z^2}{(z-b_{j,m})(z-\overline{b_{j,m}})}+\theta_{\infty,m}\frac{1-z^2}{(z-b_{m,m})(z-\overline{b_{m,m}})}.
\end{split}
\end{equation}
As the sequence of domains $\{\Omega_m\}$ kernel converges  to $\Omega_\infty$, the Carath\'eodory kernel convergence theorem implies that $\{h_m\}$ converges uniformly on compacta of $\D$ to  $h_\infty$. As a consequence, $\{p_m\}$ converges uniformly on compacta to  $p_\infty$, which, from \eqref{p-m}, it is not difficult to be seen having the form
\begin{equation}\label{forma-p}
\frac{1}{p_\infty(z)}=\frac{(1-2\theta_1)}{2}\frac{1-z}{1+z}+\sum_{j=1}^\infty (\theta_j-\theta_{j+1})\frac{1-z^2}{(z-b_{j})(z-\overline{b_{j}})},
\end{equation}
where $b_j=\lim_{m\to \infty} b_{j,m}$.  By Lemma \ref{angulo}, the semigroup $(\phi^\infty_t)$ has a sequence of boundary regular fixed points with dilation $\theta_j-\theta_{j+1}$, $j\in \N$. From this  and from the fact that the function $p_\infty$ extends meromorphic on $\C\setminus\{1\}$ and its zeros are the $b_j$'s, it follows that the $b_j$'s are actually boundary regular fixed points and also $b_j$ belongs to the arc with extremes $a_{j,\infty}$ and $a_{j+1,\infty}$ for each $j\in \N$.

Now we want to show that, for a suitable choice of $\{\theta_j\}$, the point $1$, which corresponds to the (non isolated) tip $h_\infty(1)=T_\infty$ is not a regular pole of $G_\infty(z)$. This is the case if and only if
\[
\lim_{(0,1)\ni r\to 1} \frac{1}{p_\infty(r)(1-r)}=\infty.
\]
Now, by \eqref{forma-p} and by Fatou's lemma, this condition holds if
\[
\sum_{j=1}^\infty \frac{\theta_j-\theta_{j+1}}{|1-b_j|^2}=\infty.
\]
However, by \eqref{veloc-a}
\[
\sum_{j=1}^\infty \frac{\theta_j-\theta_{j+1}}{|1-b_j|^2}\geq \sum_{j=1}^\infty \frac{\theta_j-\theta_{j+1}}{|1-a_{j,\infty}|^2}\geq \sum_{j=1}^\infty \frac{\theta_j-\theta_{j+1}}{K_\infty |1-T_j|}\simeq \sum_{j=1}^\infty \frac{\theta_j-\theta_{j+1}}{\theta_j},
\]
and the last series diverges if for instance $\theta_j=1/j$.
\end{example}


\begin{thebibliography}{99}
\bibitem{Abate} M. Abate, {\sl Iteration theory of holomorphic maps on taut manifolds}, Mediterranean Press, Rende, 1989.

\bibitem{Berkson-Porta}E. Berkson, H. Porta, \textit{Semigroups of
holomorphic functions and composition operators,} Michigan
Math. J. \textbf{25}, (1978), 101---115.

\bibitem{Be} D. Bertilsson, {\sl On Brennan's conjecture in conformal mapping}. PhD Dissertation, Royal Institute of Technology, Stockholm, 1999. Royal Institute of Technology, Stockholm, 1999.

\bibitem{BCD2} F. Bracci, M. D. Contreras, S. D\'iaz-Madrigal, {\sl Pluripotential theory,
semigroups and boundary behavior of infinitesimal generators in
strongly convex domains}. J. Eur. Math. Soc., 12, 1, (2010),
23---53.

\bibitem{BCD} F. Bracci, M. D. Contreras, S. D\'iaz-Madrigal, {\sl Evolution Families and the Loewner Equation I: the unit
disc}. J. Reine Angew. Math., to appear.

\bibitem{B} J. E. Brennan, {\sl The integrability of the derivative in conformal mapping}.
J. London Math. Soc. (2) 18 (1978), no. 2, 261--–272.

\bibitem{CM} L. Carleson, N. G.  Makarov, {\sl Some results connected with Brennan's conjecture}. Ark. Mat. 32 (1994), no. 1, 33–--62.

\bibitem{CD} M. D. Contreras, S. D\'iaz-Madrigal, {\sl Analytic flows on the unit disk: angular derivatives and boundary fixed points}. Pacific J. Math. 222, 2, (2005), 253---286.

\bibitem{CDP} M. D. Contreras, S. D\'iaz-Madrigal, Ch. Pommerenke, {\sl Fixed points and boundary behaviour of the Koenings function}. Ann. Acad. Sci. Fenn. Math. 29, (2004), 471---488.

\bibitem{CDP2} M. D. Contreras, S. D\'iaz-Madrigal, Ch. Pommerenke, {\sl On boundary critical points for semigroups of analytic functions}. Math. Scand. 98 (2006), no. 1, 125---142.

\bibitem{CMbook} C. C. Cowen, B. D.  MacCluer,{\sl Composition operators on spaces of analytic functions}. Studies in Advanced Mathematics. CRC Press, Boca Raton, FL, 1995.

\bibitem{CP} C.C. Cowen, Ch. Pommerenke, {\sl
Inequalities for the angular derivative of an analytic function
in the unit disk}. J. London Math. Soc. (2), 26, (1982), 271---289.

\bibitem{ES} M. Elin, D. Shoikhet,  {\sl Semigroups of holomorphic mappings with boundary fixed points and spirallike mappings}. Geometric function theory in several complex variables, 82–117, World Sci. Publ., River Edge, NJ, 2004.

\bibitem{EST} M. Elin, D. Shoikhet, N. Tarkhanov, {\sl Separation of boundary singularities for holomorphic generators}. Ann. Mat. Pura Appl. 190, 4, (2011), 595---618.

\bibitem{ESZ} M. Elin, D.  Shoikhet, L. Zalcman, {\sl A flower structure of backward flow invariant domains for semigroups}. Ann. Acad. Sci. Fenn. Math. 33 (2008), no. 1, 3–34.

\bibitem{H} M. Heins, {\sl Semigroups of holomorphic maps of a Riemann surface into itself which are homomorphs of the set of positive reals considered additively}. E. B. Christoffel (Aachen/Monschau, 1979), pp. 314–331, Birkhäuser, Basel-Boston, Mass., 1981.

\bibitem{Pe} E. Peschl, {\sl Zur Theorie der schlichten Funktionen}. J. Reine Angew. Math. 176,  (1936), 61---94.

\bibitem{P}Ch. Pommerenke, {\sl Univalent Functions}, Vandenhoeck \&
Ruprecht, G\"{o}ttingen, 1975.

\bibitem{P2}Ch. Pommerenke, {\sl Boundary behaviour of conformal mappings}, Springer-Verlag, 1992.

\bibitem{Shb} D. Shoikhet, {\sl Semigroups in geometrical function theory}. Kluwer Academic Publishers, Dordrecht, 2001.

\bibitem{Siskakis-tesis} A. G. Siskakis, {\sl Semigroups of Composition
Operators and the Ces\`{a}ro Operator on $H^{p}(D)$}, Ph. D.
Thesis, University of Illinois, 1985.



\end{thebibliography}
\end{document}